\documentclass[11pt]{amsart}

\usepackage{amsmath,amssymb,amsthm,mathtools}
\usepackage{fullpage}
\usepackage{enumitem}
\usepackage[colorlinks=true,citecolor=blue,linkcolor=blue,urlcolor=blue]{hyperref}

\title[Holographic functions and neural networks]{Holographic functions and neural networks}
\author{Bal\'azs Szegedy}

\address{Rényi Institute of Mathematics, Budapest, Hungary}
\email{szegedyb@gmail.com}
\date{\today}

\newtheorem{theorem}{Theorem}[section]
\newtheorem{lemma}[theorem]{Lemma}
\newtheorem{proposition}[theorem]{Proposition}

\newtheorem{definition}[theorem]{Definition}
\newtheorem{remark}[theorem]{Remark}
\newtheorem{example}[theorem]{Example}

\newcommand{\E}{\mathbb E}
\newcommand{\Pp}{\mathbb P}
\newcommand{\N}{\mathbb N}
\newcommand{\R}{\mathbb R}
\newcommand{\1}{\mathbf 1}
\newcommand{\Lip}{\operatorname{Lip}}

\newcommand{\cube}{\{0,1\}}
\newcommand{\norm}[1]{\left\lVert #1\right\rVert}
\newcommand{\abs}[1]{\left\lvert #1\right\rvert}

\begin{document}
\maketitle

\begin{abstract}
A fuzzy Boolean function is a map $f:\cube^n\to [0,1]$, where $n\in\mathbb N$.  We introduce and compare three ways of saying that such a function has bounded complexity.  The first is a sampling property: the value $f(x)$ can be recovered, up to small error and with high probability, from the values of a bounded number of randomly chosen coordinates of $x$.  We call this the holographic property.  The second is a structural property: $f$ is uniformly close to a bounded-degree polynomial in boundedly many bounded linear coordinate forms. The third is computational: $f$ is uniformly close to the output of a neural network with a bounded number of non-input neurons, bounded Lipschitz activation functions and bounded incoming weights.  We prove that these three properties are equivalent up to quantitative changes of the parameters.  The implication from holography to polynomial structure uses a variant of a weak version of hypergraph regularity. 
\end{abstract}

\section{Introduction} 
We define a fuzzy Boolean function to be a map
\[
        f:\cube^n\longrightarrow [0,1].
\]
Such functions arise naturally in many contexts.  In machine learning, for
instance, the value \(f(x)\) may be interpreted as the probability that the
input string \(x\) has a certain learned property.  Since many classification problems do not
have an intrinsically sharp yes-or-no answer, the fuzzy Boolean framework is a
natural one. Artificial neural networks have proved remarkably effective in machine learning and this motivates the following broad question.

\begin{quote}
\itshape
Is there a mathematically robust complexity theory for fuzzy Boolean functions
which reflects their representability, or approximability, by neural networks
of bounded complexity?
\end{quote}

The representability of multivariable functions by neural networks is a broad
topic, with a large literature ranging from classical universal approximation
theorems \cite{Cybenko,Hornik,Pinkus} to quantitative approximation results for
specific architectures and smoothness classes \cite{Yarotsky}, as well as
modern studies of expressive power, depth, and complexity.

What differentiates our work is a particular dimension-independent structural
viewpoint.  We study representability questions for fuzzy Boolean functions
through a sampling property and a regularization principle related to
Szemer\'edi's famous regularity lemma \cite{Szemeredi,FriezeKannan} and to the more general hypergraph
regularity theory \cite{KNRSS,Gowers,Tao}. Roughly speaking, we show that bounded random recoverability, which we call
the {\bf holographic property}, is equivalent to approximation by bounded-degree
polynomials in boundedly many coordinate linear forms with bounded
\(\ell_1\)-norms. We then relate this structural statement to representability by neural networks satisfying strong boundedness conditions. 

In this way, we isolate one particular regime in which the question above has
a rigorous and conceptually transparent answer.  In addition, our results
suggest how certain characteristic features of neural networks, such as large
weighted averages followed by local nonlinearities, can emerge naturally from
abstract complexity notions.

The guiding observation is that, in many examples related to machine learning,
the information determining \(f(x)\) is, in a suitable sense,
\emph{holographically distributed} among the bits of \(x\).  Roughly speaking, this
means that the value of \(f(x)\) can be estimated, with high probability and
small error, from the values of a bounded number of sampled coordinates,
provided that the sampled locations are also known.

At first sight this may seem counterintuitive since a bounded sample appears to
contain only a negligible fraction of the input when \(n\) is large.  However,
in many high-dimensional data sets the relevant information is highly
redundant.  For example, if one considers images of increasing resolution, then
a sufficiently large but bounded random sample of pixels may already give substantial
information about the content of the image, even though it represents only a
small fraction of all pixels.  The holographic condition is an abstract
formalization of this kind of distributed redundancy.

We also note a possible connection with statistical physics.  Macroscopic
states of a system often arise from, and depend on, a very large number of
microscopic degrees of freedom.  The notions considered here may provide a
useful language for studying when such macroscopic observables can be recovered
from limited random information about the underlying microstate.

The next definition formalizes our main concept.

\begin{definition}[Holographic functions]
Let $k\in\N$ and $\varepsilon>0$.  A function $f:\cube^n\to[0,1]$ is called \emph{$(k,\varepsilon)$-holographic} if there exist probability measures $\mu_1,\ldots,\mu_k$ on $[n]=\{1,\ldots,n\}$ and functions
\[
        f_s:\cube^k\to[0,1],\qquad s=(s_1,\ldots,s_k)\in[n]^k,
\]
such that for every $x\in\cube^n$, if $S_1,\ldots,S_k$ are independent random variables with $S_j$ having law $\mu_j$, and $S=(S_1,\ldots,S_k)$, then
\[
        \Pp\left(\abs{ f(x)-f_S(x_{S_1},\ldots,x_{S_k})}\leq \varepsilon\right)\geq 1-\varepsilon .
\]
When $S=(S_1,\ldots,S_k)$, we write $x_S$ for the tuple $(x_{S_1},\ldots,x_{S_k})$.  The functions $f_s$ are called \emph{test functions} and $\mu_1,\ldots,\mu_k$ are called \emph{sampling measures}.  
\end{definition}

The definition intentionally allows the test function to depend on the sampled locations.  This is important since the same bit value may have different meanings at different coordinates.  

\begin{remark}
Interestingly, as we will show in Section~\ref{sec:idvsnonid}, the sampling
measures in the definition above may be taken to be identical at the qualitative
level, at the cost of changing the sample bound and the accuracy parameter. For the precise definition of qualitative equivalence between properties for fuzzy Boolean functions see Definition \ref{qeq}. The present formulation, which allows different sampling measures in different
query positions, is more convenient for the proof of the main theorem.  From a
conceptual point of view, however, the formulation with a single common
sampling measure is also natural and may be preferable in some applications.
\end{remark}

The main purpose of this paper is to compare several notions of bounded
complexity for fuzzy Boolean functions.  The key organizing concept is
\emph{qualitative equivalence}, which formalizes when two such notions can be
converted into one another with dimension-independent control of complexity and
accuracy. 

A two-parameter property \(P\) of fuzzy Boolean functions assigns, to each
pair \((K\in\mathbb{N},\varepsilon>0)\), a class \(P(K,\varepsilon)\) of fuzzy Boolean
functions, over all dimensions \(n\), satisfying the corresponding property
with complexity parameter \(K\) and accuracy \(\varepsilon\).  The holographic
property defined above is an example of such a property.

\begin{definition}[Qualitative implication and equivalence]\label{qeq}
Let \(P\) and \(Q\) be two-parameter properties of fuzzy Boolean functions.
We say that \(P\) qualitatively implies \(Q\), and write
\(P\preceq Q\), if for every \(\varepsilon>0\) there exists
\(\delta=\delta(\varepsilon)>0\) such that for every \(K\in\mathbb N\) there
exists \(K'=K'(K,\varepsilon)\in\mathbb N\) satisfying
\[
        P(K,\delta)\subseteq Q(K',\varepsilon).
\]
We say that \(P\) and \(Q\) are qualitatively equivalent if
\(P\preceq Q\) and \(Q\preceq P\).
\end{definition}

Now we turn to the second property which we call {\it polynomial}.  The polynomials that appear are not arbitrary polynomials in the $n$ coordinates.  They are polynomials in boundedly many bounded linear forms in the coordinates.  The coefficient $\ell^1$-norm of a polynomial always means the sum of the absolute values of its coefficients in the standard monomial basis.

\begin{definition}[Polynomial property]
Let $K\in\mathbb N$ and $\varepsilon>0$. A function
$f:\{0,1\}^n\to[0,1]$ has the $(K,\varepsilon)$-polynomial property
if there exist an integer $m$ with $1\leq m\leq K$, linear functions
\[
        L_i(x)=\sum_{j=1}^n w_{ij}x_j
\]
with
\[
        \sum_{j=1}^n |w_{ij}|\leq K,
\]
and a polynomial $p\in\mathbb R[y_1,\ldots,y_m]$ of degree at most $K$
and coefficient $\ell^1$-norm at most $K$, such that
\[
        |f(x)-p(L_1(x),\ldots,L_m(x))|\leq\varepsilon
\]
for every $x\in\{0,1\}^n$. 
\end{definition}

The third property is neural-network representability. The following model is deliberately broad and somewhat nonstandard. The activation functions are allowed to vary from neuron to neuron, but their Lipschitz constants, ranges, and incoming affine norms are uniformly controlled by the complexity parameter.

\begin{definition}[Bounded Lipschitz neural networks]
A \emph{bounded Lipschitz neural network} on $\cube^n$ is a finite directed acyclic graph with exactly $n$ input vertices, one carrying each coordinate function $x_1,\ldots,x_n$.  Every non-input vertex $v$ carries an affine form
\[
        L_v(z)=c_v+\sum_{u\prec v} w_{v,u}z_u
\]
in the values of earlier vertices and an activation function $\sigma_v:\R\to[0,1]$.  The value at $v$ is $z_v=\sigma_v(L_v(z))$.  One distinguished non-input vertex is the output vertex.

The \emph{complexity} of the network is at most $K$ if the number of non-input vertices is at most $K$ and, for every non-input vertex $v$,
\[
        \abs{c_v}+\sum_{u\prec v}\abs{w_{v,u}}\leq K,
        \qquad
        \Lip(\sigma_v)\leq K .
\]
A non-input vertex may receive edges from any subset of the input vertices, including a subset whose size depends on $n$, provided the total incoming $\ell^1$-weight is bounded as above.  The number of nonzero incoming edges from input vertices is not counted separately.
\end{definition}

\begin{remark}
Bounded Lipschitz neural networks are not required to have a bounded total
number of parameters.  Although the number of non-input neurons is bounded, the
number of input vertices may grow with the ambient dimension.
\end{remark}

\begin{definition}[Bounded neural-network property]
Let $K\in\N$ and $\varepsilon>0$.  A function $f:\cube^n\to[0,1]$ has the \emph{$(K,\varepsilon)$-bounded neural-network property} if there exists a bounded Lipschitz neural network of complexity at most $K$ computing a function $F:\cube^n\to[0,1]$ such that
\[
        \norm{f-F}_{\infty}\leq\varepsilon .
\]
\end{definition}

Now we are ready to state the main theorem which establishes qualitative equivalence of the above three properties.

\begin{theorem}[Main theorem]\label{thm:main}
The holographic property, the polynomial property and the bounded neural-network property are qualitatively equivalent in the sense of Definition \ref{qeq}.
More explicitly:
\begin{enumerate}[label=\textup{(\alph*)}]
    \item for every $k$ and $\varepsilon\in(0,1)$ there is $K_1=K_1(k,\varepsilon)\in\N$ such that every $(k,\varepsilon)$-holographic function has the $(K_1,3\varepsilon)$-polynomial property;
    \item for every $K$ there is $K_2=K_2(K)\in\N$ such that every function with the $(K,\varepsilon)$-polynomial property has the $(K_2,\varepsilon)$-bounded neural-network property;
    \item for every $K$ and $\varepsilon\in(0,1)$ there is $K_3=K_3(K,\varepsilon)\in\N$ such that every function with the $(K,\varepsilon)$-bounded neural-network property is $(K_3,3\varepsilon)$-holographic.
\end{enumerate}
\end{theorem}

The proof of Theorem \ref{thm:main} follows directly from three major propositions proved in
the subsequent sections of the paper. Proposition~\ref{prop:holo-poly} gives part \textup{(a)},
Proposition~\ref{prop:poly-nn} gives part \textup{(b)}, and
Proposition~\ref{prop:nn-holo} gives part \textup{(c)}.  
The constants in the proof are not optimized.  The implication from holography to polynomial structure uses a regularity lemma and therefore has relatively poor quantitative dependence in general.  The other two implications have more elementary dependences once the dimension of the polynomial representation or the network complexity is fixed.  

\begin{remark}
Theorem~\ref{thm:main} establishes cyclic qualitative implication among the
three properties above.  Since qualitative implication is transitive, it follows
that the three properties are qualitatively equivalent.
\end{remark}

\section{Elementary examples}

\begin{example}[Functions of weighted coordinate averages]
Let $\mu$ be a probability measure on $[n]$ and let
\[
        f(x)=\sigma\left(\sum_{i=1}^n \mu(i)x_i\right),
\]
where $\sigma:[0,1]\to[0,1]$ is $1$-Lipschitz.  Then, for every $\varepsilon\in(0,1)$, $f$ is $(r,\varepsilon)$-holographic with
\[
        r=O(\varepsilon^{-2}\log(1/\varepsilon)).
\]
Indeed, sample $r$ independent coordinates from $\mu$ and use the test function
\[
        f_s(\alpha_1,\ldots,\alpha_r)
        =
        \sigma\left(\frac1r\sum_{j=1}^r\alpha_j\right).
\]
The empirical average of the sampled bits approximates $\sum_i\mu(i)x_i$ with probability at least $1-\varepsilon$, uniformly in $x$, by Hoeffding's inequality.
\end{example}

\begin{example}[Small juntas]
If $f$ depends on at most $r$ coordinates $i_1,\ldots,i_r$, then $f$ is holographic with $k=r$ by taking $\mu_j$ to be the point mass at $i_j$ and using the test function that reads these coordinates.  In the more restrictive iid version of the definition one would need a coupon-collector bound, for example $k=O(r\log(r/\varepsilon))$ samples from the uniform distribution on the relevant coordinates.
\end{example}

\begin{example}[A non-example]
It is easy to see that the parity function $x\mapsto x_1+\cdots+x_n\pmod 2$ is not holographic with bounded $k$ and error below $1/2$ as $n\to\infty$. A stronger form of this obstruction follows from Lemma~\ref{lem:test-average}:
if \(n\) is large compared with \(k\), then most coordinates must have small
influence on any \((k,\varepsilon)\)-holographic function.  Thus highly
sensitive functions such as parity cannot satisfy the holographic property with
bounded sample complexity.
\end{example}

\section{From holography to polynomials}

We begin with a simple observation.  If a holographic test succeeds with high probability, then its expectation gives a uniformly good approximation.

\begin{lemma}[Averaging the tests]\label{lem:test-average}
Let $f:\cube^n\to[0,1]$ be $(k,\varepsilon)$-holographic with sampling measures $\mu_1,\ldots,\mu_k$ and test functions $f_s$.  Define
\[
        F(x)=\E f_S(x_{S_1},\ldots,x_{S_k}),
\]
where $S_1,\ldots,S_k$ are independent and $S_j$ has law $\mu_j$.  Then
\[
        \norm{f-F}_{\infty}\leq 2\varepsilon .
\]
\end{lemma}

\begin{proof}
Fix $x$.  With probability at least $1-\varepsilon$, the random variable $f_S(x_S)$ differs from $f(x)$ by at most $\varepsilon$.  On the exceptional event it differs by at most $1$.  Hence
\[
        \abs{F(x)-f(x)}\leq (1-\varepsilon)\varepsilon+\varepsilon\leq2\varepsilon .
\]
\end{proof}

For each pattern $\alpha=(\alpha_1,\ldots,\alpha_k)\in\cube^k$, define a coefficient array
\[
        g_\alpha(s)=f_s(\alpha),\qquad s\in[n]^k.
\]
Then
\[
        F(x)=\sum_{\alpha\in\cube^k}\E
        \left[g_\alpha(S)\prod_{j=1}^k \1_{x_{S_j}=\alpha_j}\right].
\]
The problem is that the arrays $g_\alpha$ may be arbitrary functions on $[n]^k$.  The following regularity lemma says that, for the purpose of integration over product boxes, all these arrays can be simultaneously approximated by arrays that are constant on the cells of a bounded partition.

\begin{theorem}[Box regularity lemma]\label{thm:box-regularity}
For every $\eta>0$ and $k,t\in\N$ there exists $M=M(k,t,\eta)\in\N$ with the following property.  Let $\mu_1,\ldots,\mu_k$ be probability measures on $[n]$ and let
\[
        g_1,\ldots,g_t:[n]^k\to[0,1].
\]
Then there is a partition $[n]=T_1\cup\cdots\cup T_m$ with $m\leq M$ such that for every $i\in\{1,\ldots,t\}$ there is a step array
\[
        W_i:[m]^k\to[0,1]
\]
for which, for all product sets $A_1\times\cdots\times A_k\subseteq[n]^k$,
\[
\left|
\E\left[g_i(S)\prod_{j=1}^k\1_{S_j\in A_j}\right]
-
\sum_{u\in[m]^k} W_i(u)\prod_{j=1}^k\mu_j(A_j\cap T_{u_j})
\right|
\leq \eta ,
\]
where $S_1,\ldots,S_k$ are independent and $S_j$ has law $\mu_j$.  One may take, for instance,
\[
        M(k,t,\eta)=2^{k\lceil t\eta^{-2}\rceil}.
\]
\end{theorem}

\begin{remark}
Theorem~\ref{thm:box-regularity} is a standard weak cut-norm, or weak hypergraph-regularity, statement for bounded measurable $k$-arrays tested against product boxes; it is in the same spirit as weak graph and hypergraph regularity results such as those in \cite{FriezeKannan,Gowers,Tao,RodlSchacht}.  A direct simultaneous energy-increment proof, allowing the coordinate measures $\mu_1,\ldots,\mu_k$ to be distinct while using one common partition of $[n]$, is included in Section~\ref{sec:regularization}.  Equivalently, the conditional expectations used there are taken with respect to the product measure $\mu_1\times\cdots\times\mu_k$, so distinct coordinate measures cause no additional difficulty.  Cells of measure zero cause no ambiguity in the displayed formula, since all terms involving such cells are multiplied by zero.
\end{remark}

We now prove the first implication in the main theorem.

\begin{proposition}[Holography implies polynomial structure]\label{prop:holo-poly}
For every $k\in\N$ and $\varepsilon\in(0,1)$ there exists $K=K(k,\varepsilon)\in\N$ such that every $(k,\varepsilon)$-holographic function has the $(K,3\varepsilon)$-polynomial property.
\end{proposition}

\begin{proof}
Let $f:\{0,1\}^n\to [0,1]$ be $(k,\varepsilon)$-holographic.  Let $\mu_1,\ldots,\mu_k$ and $f_s$ be as in the definition, and let $F$ be the average test function from Lemma~\ref{lem:test-average}.  For each $\alpha\in\cube^k$, let $g_\alpha(s)=f_s(\alpha)$.  Apply Theorem~\ref{thm:box-regularity} to the $t=2^k$ arrays $g_\alpha$ with parameter
\[
        \eta=\varepsilon/2^k.
\]
We obtain a partition $[n]=T_1\cup\cdots\cup T_q$, $q\leq M(k,2^k,\eta)$, and step arrays $W_\alpha:[q]^k\to[0,1]$.

For $\ell\in[k]$ and $u\in[q]$, write
\[
        a_{\ell, u}=\mu_\ell(T_u),\qquad
        B_{\ell, u}(x)=\sum_{i\in T_u}\mu_\ell(i)x_i.
\]
Note that $a_{\ell, u}$ is a constant and $B_{\ell, u}$ is a linear function in the variable vector $x$. Each linear form $B_{\ell, u}$ has coefficient $\ell^1$-norm at most $1$ and $0\leq a_{\ell, u}\leq 1$ holds for every $\ell,u$. The $kq$ functions \(B_{\ell, u}\) are precisely the linear forms that will be counted in the polynomial representation, and their number will be absorbed into the final complexity parameter \(K\). Define the linear polynomials
\[
        C_{\ell,u,1}=B_{\ell, u},\qquad C_{\ell,u,0}=a_{\ell, u}-B_{\ell, u}
\]
and the polynomial $P$ in the $kq$ variables $(B_{\ell u})_{\ell,u}$ by
\[
        P\bigl((B_{\ell, u})_{\ell,u}\bigr)
        =\sum_{\alpha\in\cube^k}\sum_{u\in[q]^k} W_\alpha(u)
        \prod_{j=1}^k C_{j,u_j,\alpha_j}.
\]

Note that there is a slight abuse of notation in the definition of \(P\) and
\(C_{\ell,u,\beta}\).  They are treated as polynomials in the formal variables
\((B_{\ell,u})_{\ell,u}\), and their coefficient norms are computed in this
polynomial ring.  On the other hand, each \(B_{\ell,u}\) represents a linear
form on the coordinates of \(x\in\{0,1\}^n\).  Thus, after substituting these
linear forms into \(P\), we may also evaluate \(P\) on Boolean vectors \(x\).

The polynomial $P$ has degree at most $k$.  Its coefficient $\ell^1$-norm is bounded explicitly by
\[
        \norm{P}_{\ell^1}
        \leq
        \sum_{\alpha\in\cube^k}\sum_{u\in[q]^k} W_\alpha(u)
        \prod_{j=1}^k \norm{C_{j,u_j,\alpha_j}}_{\ell^1}
        \leq 2^k q^k\cdot 2^k
        =4^k q^k.
\]

For a fixed $x$, the set of indices with bit value $\alpha_j$ is
\[
        A_j(\alpha_j,x)=\{i\in[n]:x_i=\alpha_j\}.
\]
The regularity lemma gives, for each $\alpha$,
\[
\left|
\E\left[g_\alpha(S)\prod_{j=1}^k \1_{x_{S_j}=\alpha_j}\right]
-
\sum_{u\in[q]^k} W_\alpha(u)\prod_{j=1}^k C_{j,u_j,\alpha_j}(x)
\right|\leq\eta .
\]
Summing over the $2^k$ patterns $\alpha$ shows that
\[
        \norm{F-P((B_{\ell, u})_{\ell,u})}_{\infty}\leq 2^k\eta=\varepsilon .
\]
Together with Lemma~\ref{lem:test-average}, this gives
\[
        \norm{f-P((B_{\ell, u})_{\ell,u})}_{\infty}\leq3\varepsilon.
\]
Taking $K$ to be $4^k M(k,2^k,\eta)^k$ gives a complexity parameter depending only on $k$ and $\varepsilon$.  The $kq$ linear functions $B_{\ell, u}$, together with the polynomial $P$, therefore satisfy the formal polynomial property.
\end{proof}

\section{From polynomials to neural networks}

The second implication is a finite-dimensional algebraic realization statement.  Once a function is expressed as a polynomial in boundedly many bounded linear forms, a bounded-size neural network with bounded weights and bounded Lipschitz constants can compute those linear forms and then compute the clipped polynomial.  

\begin{lemma}[Exact neural representation of polynomials]\label{lem:universal}
Let \(d,D\in\mathbb N\) and \(C\geq1\).  There exists
\(R=R(d,D,C)\in\mathbb N\) such that for every polynomial
\(Q\in\mathbb R[y_1,\ldots,y_d]\) of degree at most \(D\) and coefficient
\(\ell^1\)-norm at most \(C\), the function
\[
        y\longmapsto \chi(Q(y)),\qquad
        \chi(t)=\min(1,\max(0,t)),
\]
on \([0,1]^d\) is computed exactly by a bounded Lipschitz neural network of
complexity at most \(R\) in the auxiliary real-input model.
\end{lemma}

\begin{proof}
Let
\[
        \chi(t)=\min(1,\max(0,t)),
        \qquad
        \psi(t)=\chi(t)^2 .
\]
Then \(\chi,\psi:\mathbb R\to[0,1]\) are Lipschitz, with Lipschitz constants at
most \(1\) and \(2\), respectively, and on \([0,1]\) one has
\(\chi(t)=t\) and \(\psi(t)=t^2\). We define a multiplication module consisting of four neurons. Suppose two earlier vertices,
or input vertices, carry values \(u,v\in[0,1]\).  Add three square vertices
carrying
\[
        a=\psi\left(\frac{u+v}{2}\right),\qquad
        b=\psi(u),\qquad
        c=\psi(v),
\]
and then add a fourth vertex with activation \(\chi\) and linear input
\[
        2a-\frac12 b-\frac12 c .
\]
Since \(u,v\in[0,1]\), this fourth vertex carries
\[
        \chi\left(2\left(\frac{u+v}{2}\right)^2-\frac12 u^2-\frac12 v^2\right)
        =
        \chi(uv)=uv .
\]
All incoming linear \(\ell^1\)-norms are at most \(3\), and all
activation Lipschitz constants are at most \(2\).

Using the multiplication module iteratively, every monomial
\[
        y^\alpha=y_1^{\alpha_1}\cdots y_d^{\alpha_d},
        \qquad |\alpha|\leq D,
\]
can be computed exactly by at most \(4(|\alpha|-1)\) non-input vertices when
\(|\alpha|\geq2\). Note that monomials of degree \(0\) and \(1\) require no multiplication
vertices, since the constant term is handled by the final affine input and the
degree-one monomials are input coordinates.  At each stage the computed value
lies in \([0,1]\), so the next multiplication module applies.

There are only
\[
        N(d,D)=\binom{d+D}{D}
\]
monomials of degree at most \(D\).  Compute all monomials of degree at least
\(2\) by the iterative construction above, in any fixed order.  Finally add one
output vertex with activation \(\chi\) whose affine input is the linear
combination of the available monomial values with the coefficients of \(Q\).
The constant coefficient of \(Q\) is used as the bias.  The affine
\(\ell^1\)-norm at this output vertex is at most the coefficient
\(\ell^1\)-norm of \(Q\), hence at most \(C\).  The output is exactly
\(\chi(Q(y))\) on \([0,1]^d\).

Thus the number of non-input vertices is bounded by a function of \(d\) and
\(D\), for instance
\[
        1+4(D-1)N(d,D),
\]
and every incoming affine \(\ell^1\)-norm and activation Lipschitz constant is
bounded by \(\max(C,3)\).  Taking \(R\) to be the ceiling of the maximum of
these bounds proves the lemma.
\end{proof}

\begin{proposition}[Polynomial structure implies neural representation]\label{prop:poly-nn}
For every \(K\in\N\) there exists \(K'=K'(K)\in\N\) such that every function
with the \((K,\varepsilon)\)-polynomial property has the
\((K',\varepsilon)\)-bounded neural-network property, for every
\(\varepsilon\in(0,1)\).
\end{proposition}

\begin{proof}
Let $f:\cube^n\to[0,1]$ have the $(K,\varepsilon)$-polynomial property.  Thus there are $m\leq K$, linear functions
\[
        L_i(x)=\sum_{j=1}^n w_{ij}x_j,\qquad
        \sum_{j=1}^n |w_{ij}|\leq K,
\]
and a polynomial $p$ of degree at most $K$ and coefficient $\ell^1$-norm at most $K$ such that
\[
        |f(x)-p(L_1(x),\ldots,L_m(x))|\leq\varepsilon
\]
for every $x\in\cube^n$.

For $x\in\cube^n$ we have $L_i(x)\in[-K,K]$.  Define scaled affine forms
\[
        Y_i(x)=\frac{L_i(x)+K}{2K}
        =
        \frac{1}{2}+\sum_{j=1}^n \frac{w_{ij}}{2K}x_j .
\]
Then $Y_i(x)\in[0,1]$ on $\cube^n$, and
\[
        \frac{1}{2}+\sum_{j=1}^n\left|\frac{w_{ij}}{2K}\right|
        \leq
        \frac{K+\sum_j |w_{ij}|}{2K}\leq1.
\]
Hence the first layer of a network may compute the $m$ values $Y_i$ exactly on $\cube^n$ using the clipped identity activation $\chi(t)=\min(1,\max(0,t))$.  These first-layer vertices may receive edges from all \(n\) input vertices; only the total incoming \(\ell^1\)-weight is relevant to the complexity.

Define a polynomial $Q$ on $[0,1]^m$ by
\[
        Q(y_1,\ldots,y_m)
        =
        p(2Ky_1-K,\ldots,2Ky_m-K).
\]
It has degree at most $K$.  Its coefficient $\ell^1$-norm is bounded only in terms of $K$. Indeed, each monomial of degree at most $K$ in the variables $L_i=2Ky_i-K$ expands with coefficient $\ell^1$-norm at most $(3K)^K$, and the coefficient $\ell^1$-norm of $p$ is at most $K$.  Thus
\[
        \norm{Q}_{\ell^1}\leq K(3K)^K .
\]
Apply Lemma~\ref{lem:universal} with
\[
        d=m,\qquad D=K,\qquad C=K(3K)^K .
\]
This gives an auxiliary real-input network, of complexity bounded in terms of $K$ only, which computes
\[
        \chi(Q(y))=\min(1,\max(0,Q(y)))
\]
exactly on \([0,1]^m\).

Composing this real-input network with the first layer computing $Y_1,\ldots,Y_m$ gives a Boolean-input bounded Lipschitz network whose complexity is bounded by a number depending only on $K$.  The outputs of the first layer lie in $[0,1]^m$ for every Boolean input $x$, so the composed network evaluates the auxiliary real-input network only on the domain on which Lemma~\ref{lem:universal} gives exact correctness.  For every $x\in\cube^n$ the output is
\[
        \chi(Q(Y_1(x),\ldots,Y_m(x)))
        =
        \chi(p(L_1(x),\ldots,L_m(x))).
\]
Since $f(x)\in[0,1]$ and $\chi$ is the metric projection of $\R$ onto $[0,1]$,
\[
        \abs{f(x)-\chi(t)}\leq \abs{f(x)-t}
        \qquad\text{for all }t\in\R.
\]
Therefore
\[
        \abs{f(x)-\chi(p(L_1(x),\ldots,L_m(x)))}\leq\varepsilon
\]
for every $x$. Choose $K'$ to dominate simultaneously the $m\leq K$ first-layer vertices, the number of non-input vertices in the finite-dimensional realization, all incoming affine $\ell^1$-norms in both parts of the composed network, and all activation Lipschitz constants.  This $K'$ depends only on $K$, not on $\varepsilon$.  This proves the proposition.
\end{proof}

\section{From neural networks to holography}

We now prove that bounded Lipschitz networks are holographic.  The proof is an induction over the computational graph.  The basic fact is that a bounded affine average of many bounded quantities can be estimated from a bounded number of random samples.

\begin{lemma}[Sampling a bounded linear form]\label{lem:sample-affine}
Let $B\geq1$, let $z_1,\ldots,z_N\in[0,1]$ and let
\[
        L=c+\sum_{i=1}^N w_i z_i,
        \qquad
        \abs c+\sum_i\abs{w_i}\leq B.
\]
For every $\delta,\rho\in(0,1)$ there is $r=r_{\mathrm{aff}}(B,\delta,\rho)$ and a randomized procedure which queries at most $r$ of the values $z_i$ and outputs $\widehat L$ such that
\[
        \Pp(\abs{\widehat L-L}\leq\delta)\geq1-\rho.
\]
The distribution of the queried indices depends only on the coefficients $w_i$ and not on the values $z_i$.
\end{lemma}

\begin{proof}
Let $A=\sum_i\abs{w_i}$.  If $A=0$, then $L=c$ and no query is needed.  Otherwise sample indices $I_1,\ldots,I_r$ independently with
\[
        \Pp(I_j=i)=\abs{w_i}/A.
\]
Set
\[
        Y_j=A\,\operatorname{sgn}(w_{I_j})z_{I_j},
        \qquad
        \widehat L=c+\frac1r\sum_{j=1}^rY_j.
\]
Then $\E\widehat L=L$ and $Y_j\in[-B,B]$.  Hoeffding's inequality gives
\[
        \Pp(\abs{\widehat L-L}>\delta)\leq2\exp\left(-\frac{r\delta^2}{2B^2}\right).
\]
Choosing $r\geq2B^2\delta^{-2}\log(2/\rho)$ proves the lemma.
\end{proof}

\begin{lemma}[Sampling a network]\label{lem:sample-network}
Let $F:\cube^n\to[0,1]$ be computed by a bounded Lipschitz neural network of complexity at most $K$.  For every $\delta\in(0,1)$ there exists $r=r(K,\delta)\in\N$, probability measures $\nu_1,\ldots,\nu_r$ on $[n]$, and a reconstruction function
\[
        \Phi_s:\cube^r\to[0,1]\qquad (s\in[n]^r)
\]
such that, if $S_1,\ldots,S_r$ are independent and $S_j$ has law $\nu_j$, then, for every $x\in\cube^n$,
\[
        \Pp\left(\abs{\Phi_S(x_{S_1},\ldots,x_{S_r})-F(x)}\leq\delta\right)\geq1-\delta .
\]
After the sampled locations $s$ are fixed, $\Phi_s$ is deterministic; there is no additional internal randomness.
\end{lemma}

\begin{proof}
Order the non-input vertices as \(v_1,\ldots,v_M\), where \(M\leq K\), so that
all non-input vertices feeding into \(v_j\) occur among
\(v_1,\ldots,v_{j-1}\).  Such an ordering exists because the network graph is
acyclic. We prove by induction on this order that, for every non-input vertex $v_j$ and every pair of parameters $\tau,\rho\in(0,1)$, the value $z_{v_j}(x)$ can be estimated to accuracy $\tau$ with failure probability at most $\rho$ using a number of independent coordinate samples bounded by a function $R_j(\tau,\rho)$ depending only on $K,\tau,\rho$ and $j$.  The estimator is specified by a fixed finite list of sampling measures on $[n]$, depending only on the network and on $\tau,\rho$, and by a deterministic reconstruction function depending on the sampled locations and sampled bits.  Repeated independent estimates of the same predecessor are allowed. 

Consider a vertex $v_j$.  Write its affine input as
\[
        L_{v_j}=c_{v_j}+\sum_{i=1}^n \alpha_i x_i+\sum_{\substack{u\prec v_j\\ u\ \mathrm{noninput}}} \beta_u z_u .
\]
The sum of the absolute values of all displayed coefficients, together with $\abs{c_{v_j}}$, is at most $K$.  Put
\[
        \gamma=\frac{\tau}{2K}.
\]
First estimate the direct input affine form
\[
        c_{v_j}+\sum_{i=1}^n \alpha_i x_i
\]
by Lemma~\ref{lem:sample-affine}, with $B=K$, to accuracy $\gamma$ and failure probability $\rho/2$.  This uses at most $r_{\mathrm{aff}}(K,\gamma,\rho/2)$ independent coordinate samples from the distribution proportional to $\abs{\alpha_i}$, unless all $\alpha_i$ vanish, in which case the direct input affine form is the known constant $c_{v_j}$ and no such samples are needed.

If $v_j$ has no non-input predecessors, this direct estimate, followed by the activation $\sigma_{v_j}$, gives an estimate of $z_{v_j}$; on the success event the error is at most $K\gamma\leq\tau$.  This is the base case.

In general, for each non-input predecessor $u=v_h$ with $\beta_u\ne0$, use the induction hypothesis to estimate $z_u$ to accuracy $\gamma/K$ and failure probability at most $\rho/(2K)$.  Zero-coefficient predecessors are ignored, and the number of nonzero non-input predecessors is at most the total number of non-input vertices, hence at most $K$.  Use independent samples for the different predecessor estimators and for the direct input estimator.  The complete list of samples for $v_j$ is the concatenation, in a predetermined order, of these independent lists.  Its sampling measures are therefore fixed in advance and depend only on the network and on the accuracy parameters, not on $x$ or on any sampled values.  Once all sampled locations in this concatenated list are fixed, each lower-level reconstruction is deterministic by induction, and the final reconstruction at $v_j$ is deterministic as well.

On the event that all predecessor estimates and the direct input estimate succeed, the affine form $L_{v_j}$ is estimated within at most
\[
        \gamma+\sum_{\substack{u\prec v_j\\ u\ \mathrm{noninput}}}|\beta_u|\,\frac{\gamma}{K}
        \leq \gamma+\gamma=2\gamma .
\]
Since $\Lip(\sigma_{v_j})\leq K$, the value $z_{v_j}=\sigma_{v_j}(L_{v_j})$ is then estimated within $2K\gamma=\tau$.  The union bound gives failure probability at most
\[
        \rho/2+K\cdot\frac{\rho}{2K}=\rho .
\]
The sample count satisfies the recursive bound
\[
        R_j(\tau,\rho)
        \leq
        r_{\mathrm{aff}}\!\left(K,\gamma,\frac{\rho}{2}\right)
        +
        \sum_{\substack{u=v_h\prec v_j\\ u\ \mathrm{noninput}\\ \beta_u\ne0}}
        R_h\!\left(\frac{\gamma}{K},\frac{\rho}{2K}\right),
\]
with the sum empty in the base case.  Since there are at most $K$ non-input vertices, this recursion depends only on $K,\tau,\rho$ and $j$.  The case of zero direct input weights is included by allowing the corresponding affine sample count to be zero.

Applying this construction to the output vertex with $\tau=\rho=\delta$ gives a finite list of coordinate-sampling measures of length at most a number depending only on $K$ and $\delta$.  Choose a final length
\[
        r=r(K,\delta)\geq1
\]
dominating this bound.  If fewer than $r$ samples are used in the resulting fixed list, we pad the list with dummy samples from an arbitrary fixed coordinate distribution on $[n]$ and ignore them. This is possible since $n\geq1$.  Finally, the reconstruction is clipped to $[0,1]$, which cannot increase the error because $F(x)\in[0,1]$.
\end{proof}

\begin{proposition}[Neural representation implies holography]\label{prop:nn-holo}
For every $K\in\N$ and $\varepsilon\in(0,1)$ there exists $k=k(K,\varepsilon)\in\N$ such that every function with the $(K,\varepsilon)$-bounded neural-network property is $(k,3\varepsilon)$-holographic.
\end{proposition}

\begin{proof}
Let $f$ be within $\varepsilon$ of a network output $F$ of complexity at most $K$.  Apply Lemma~\ref{lem:sample-network} to $F$ with $\delta=\varepsilon$.  We obtain $k=k(K,\varepsilon)$ independent coordinate samples, not necessarily identically distributed, and reconstruction functions $\Phi_s:\cube^k\to[0,1]$.  Define the holographic test functions by
\[
        f_s=\Phi_s,\qquad s\in[n]^k .
\]
Then, for every $x$,
\[
        \Pp(\abs{f_S(x_S)-F(x)}\leq\varepsilon)\geq1-\varepsilon.
\]
Since $\norm{f-F}_\infty\leq\varepsilon$, the same test differs from $f(x)$ by at most $2\varepsilon$ with probability at least $1-\varepsilon$.  By monotonicity in the error parameter, $2\varepsilon\leq3\varepsilon$ and $1-\varepsilon\geq1-3\varepsilon$, so $f$ is $(k,3\varepsilon)$-holographic.
\end{proof}

\section{Regularization in detail}\label{sec:regularization}

For completeness, we include a proof of the box regularity lemma in a form sufficient for this paper.  This section may be skipped by readers who are willing to use any standard weak hypergraph regularity lemma.

For a finite set $X$ with probability measures $\mu_1,\ldots,\mu_k$ and a function $h:X^k\to\R$, define the box norm
\[
        \norm{h}_{\square,k}
        =\sup_{A_1,\ldots,A_k\subseteq X}
        \left|\int_{X^k}h(x_1,\ldots,x_k)
        \prod_{j=1}^k\1_{A_j}(x_j)\,
        d\mu_1(x_1)\cdots d\mu_k(x_k)\right|.
\]
If $\mathcal P$ is a partition of $X$, let $\E(h\mid\mathcal P^k)$ be the conditional expectation of $h$ with respect to the product partition and the product measure $\mu_1\times\cdots\times\mu_k$.  On boxes of product measure zero the value of this conditional expectation may be chosen arbitrarily.

\begin{lemma}[Simultaneous weak box regularity]\label{lem:weak-box}
For every $k,t\in\N$, every $\eta>0$, and every $M_0\in\N$, there is $M=M(k,t,\eta,M_0)\in\N$ such that for every finite set $X$, every choice of probability measures $\mu_1,\ldots,\mu_k$ on $X$, every $h_1,\ldots,h_t:X^k\to[0,1]$, and every initial partition $\mathcal P_0$ of $X$ into at most $M_0$ parts, there is a refinement $\mathcal P$ of $\mathcal P_0$ into at most $M$ parts with
\[
        \norm{h_i-\E(h_i\mid\mathcal P^k)}_{\square,k}\leq\eta
        \qquad\text{for every }i\in\{1,\ldots,t\}.
\]
One may take, for instance,
\[
        M=M_0\,2^{k\lceil t\eta^{-2}\rceil}.
\]
\end{lemma}

\begin{proof}
Start with the initial partition $\mathcal P_0$.  For a partition $\mathcal P$, define the total energy
\[
        \mathcal E(\mathcal P)=\sum_{i=1}^t
        \norm{\E(h_i\mid\mathcal P^k)}_2^2,
\]
where the $L^2$ norm is taken with respect to $\mu_1\times\cdots\times\mu_k$.  Since $0\leq h_i\leq1$, we have $0\leq\mathcal E(\mathcal P)\leq t$.

If the current partition $\mathcal P$ does not satisfy the conclusion, then for some $i$ there are sets $A_1,\ldots,A_k\subseteq X$ such that, with
\[
        r=h_i-\E(h_i\mid\mathcal P^k),
        \qquad
        \phi=\prod_{j=1}^k\1_{A_j},
\]
we have
\[
        \abs{\langle r,\phi\rangle}>\eta .
\]
In particular $\norm{\phi}_2\ne0$, so the division by $\norm{\phi}_2$ below is legitimate.  Refine $\mathcal P$ by all the sets $A_1,\ldots,A_k$, and call the resulting partition $\mathcal Q$.  Then $\phi$ is $\mathcal Q^k$-measurable.  Since $\norm{\phi}_2\leq1$,
\[
        \norm{\E(r\mid\mathcal Q^k)}_2
        \geq
        \frac{\abs{\langle \E(r\mid\mathcal Q^k),\phi\rangle}}{\norm{\phi}_2}
        =
        \frac{\abs{\langle r,\phi\rangle}}{\norm{\phi}_2}
        \geq \eta .
\]
Because $\mathcal Q$ refines $\mathcal P$, the martingale difference identity gives
\[
        \norm{\E(h_i\mid\mathcal Q^k)}_2^2
        -
        \norm{\E(h_i\mid\mathcal P^k)}_2^2
        =
        \norm{\E(r\mid\mathcal Q^k)}_2^2
        \geq \eta^2 .
\]
The energies of the other functions do not decrease under refinement.  Hence the total energy increases by at least $\eta^2$.

The process therefore stops after at most $\lceil t\eta^{-2}\rceil$ refinement steps.  At each step refining by $k$ sets multiplies the number of atoms by at most $2^k$.  Thus the final number of atoms is at most $M_0\,2^{k\lceil t\eta^{-2}\rceil}$.
\end{proof}

\begin{proof}[Proof of Theorem~\ref{thm:box-regularity}]
Apply Lemma~\ref{lem:weak-box} to $g_1,\ldots,g_t$ starting with the trivial partition.  We obtain a common partition $\mathcal P=\{T_1,\ldots,T_m\}$ with
\[
        m\leq M(k,t,\eta,1)=2^{k\lceil t\eta^{-2}\rceil}
\]
and
\[
        \norm{g_i-\E(g_i\mid\mathcal P^k)}_{\square,k}\leq\eta
\]
for all $i$.  Writing $W_i(u)$ for the value of the conditional expectation on the box $T_{u_1}\times\cdots\times T_{u_k}$ gives the desired formula.  Since each $g_i$ is $[0,1]$-valued, these conditional expectations may be chosen in $[0,1]$.  If such a box has product measure zero, choose $W_i(u)$ arbitrarily in $[0,1]$; its contribution to the formula is zero.  Indeed, if $\prod_{j=1}^k\mu_j(T_{u_j})=0$, then for every product test set $A_1\times\cdots\times A_k$ one has
\[
        \prod_{j=1}^k\mu_j(A_j\cap T_{u_j})=0,
\]
so the corresponding summand is unaffected by the chosen value of $W_i(u)$.
\end{proof}

\section{Equivalence of identical and non-identical sampling}\label{sec:idvsnonid}

In the definition of holographic functions used above, the different
query positions are allowed to have different sampling measures $\mu_1,\ldots,\mu_k$ on the coordinate set \([n]\).  In this section we show that, at the qualitative
level (see Definition \ref{qeq}), this is equivalent to the apparently more restrictive version in which
all query positions are sampled independently from one common measure.
The identically sampled version is plainly a special case of the
non-identically sampled version.  The nontrivial direction is that different
measures can be simulated by taking more samples from their average.

\begin{definition}[Identically sampled holography]
A function \(f:\cube^n\to[0,1]\) is called \emph{identically sampled
\((k,\varepsilon)\)-holographic} if it satisfies the definition of
\((k,\varepsilon)\)-holography with sampling measures
\(\mu_1=\cdots=\mu_k=\mu\) for some single probability measure \(\mu\) on
\([n]\).
\end{definition}

\begin{proposition}\label{prop:identical-nonidentical-equivalence}
Let \(k\in\mathbb N\) and \(0<\varepsilon<1\).  Put
\[
        \alpha=\frac{\varepsilon^2}{4},
        \qquad
        \eta=\frac{\varepsilon^2}{2}.
\]
Suppose that \(f:\{0,1\}^n\to[0,1]\) is
\((k,\alpha)\)-holographic with non-identical sampling.  Then \(f\) is
identically sampled \((r,\varepsilon)\)-holographic for every integer
\(r\) satisfying
\[
        k e^{-r/k}\leq \eta .
\]
In particular, it is enough to take
\[
        r\geq k\log\frac{k}{\eta}
        =
        k\log\frac{2k}{\varepsilon^2}.
\]
Consequently, the identically sampled and non-identically sampled versions of
holography define the same qualitative notion after changing the complexity
and accuracy parameters.
\end{proposition}

\begin{proof}
Assume that \(f\) is \((k,\alpha)\)-holographic with non-identical sampling.
Thus there are probability measures
\[
        \mu_1,\ldots,\mu_k
\]
on \([n]\), and test functions
\[
        f_s:\{0,1\}^k\to[0,1],
        \qquad s\in[n]^k,
\]
such that for every \(x\in\{0,1\}^n\),
\[
        \mathbb P\left(
        \left|f(x)-f_S(x_{S_1},\ldots,x_{S_k})\right|
        \leq \alpha
        \right)
        \geq 1-\alpha ,
\]
where \(S_i\) are independent and \(S_i\sim\mu_i\).

Define the average measure
\[
        \mu=\frac1k\sum_{i=1}^k \mu_i .
\]
A sample from \(\mu\) may be generated in two steps: first choose a hidden label
\(I\in[k]\) uniformly, and then sample from \(\mu_I\).  More precisely, let
\[
        I_1,\ldots,I_r
\]
be independent uniform random variables on \([k]\), and, conditional on these
labels, let
\[
        T_j\sim \mu_{I_j}
\]
independently for \(j=1,\ldots,r\).  Then the variables \(T_1,\ldots,T_r\) are
independent and each has law \(\mu\).

Let \(E\) be the event that every label \(1,\ldots,k\) appears among
\(I_1,\ldots,I_r\).  By the union bound,
\[
        \mathbb P(E^c)
        \leq
        k\left(1-\frac1k\right)^r
        \leq
        k e^{-r/k}
        \leq \eta .
\]

On the event \(E\), define \(\tau_i\) to be the first index \(j\) such that
\(I_j=i\).  Thus
\[
        I_{\tau_i}=i,
        \qquad i=1,\ldots,k.
\]
The selected coordinates
\[
        T_{\tau_1},\ldots,T_{\tau_k}
\]
are independent and have respective laws
\[
        \mu_1,\ldots,\mu_k.
\]
Indeed, conditional on any label sequence for which \(E\) holds, the selected
coordinate \(T_{\tau_i}\) is sampled from \(\mu_i\), and the selected
coordinates for distinct labels are independent.

For a fixed input \(x\), define the labelled reconstruction
\[
        Y
        =
        f_{(T_{\tau_1},\ldots,T_{\tau_k})}
        (x_{T_{\tau_1}},\ldots,x_{T_{\tau_k}})
\]
on the event \(E\).  On \(E^c\), define \(Y=0\), say.  Since on \(E\) the
selected coordinates have the same joint distribution as the original
non-identically sampled tuple, we have
\[
        \mathbb P\left(|Y-f(x)|>\alpha\right)
        \leq
        \alpha+\eta .
\]
Because \(Y\in[0,1]\) and \(f(x)\in[0,1]\), it follows that
\[
        \mathbb E |Y-f(x)|
        \leq
        \alpha+\mathbb P(|Y-f(x)|>\alpha)
        \leq
        2\alpha+\eta .
\]

So far the reconstruction used the hidden labels \(I_1,\ldots,I_r\).  The
definition of identical sampling, however, only permits the test function to
depend on the observed locations \(T_1,\ldots,T_r\) and on the observed bits.
We now remove the labels by conditional expectation.  This posterior averaging
is only part of the construction of the deterministic test functions; the
eventual identical-sampling procedure itself samples only the locations from
the single measure \(\mu\).

For \(t=(t_1,\ldots,t_r)\in[n]^r\), define a probability distribution on label
tuples \(i=(i_1,\ldots,i_r)\in[k]^r\) by the posterior rule
\[
        \mathbb P(I_1=i_1,\ldots,I_r=i_r\mid T_1=t_1,\ldots,T_r=t_r)
        =
        \prod_{j=1}^r
        \frac{\mu_{i_j}(t_j)}
        {\sum_{\ell=1}^k \mu_\ell(t_j)}
\]
whenever the denominators are nonzero.  This is just Bayes' formula with
respect to the auxiliary hidden-label construction: the factor
\((1/k)\mu_{i_j}(t_j)\) in the joint law is divided by the mixture mass
\(\mu(t_j)=(1/k)\sum_{\ell=1}^k\mu_\ell(t_j)\), so the factors \(1/k\) cancel.
If a denominator is zero, the value is irrelevant because such a location is
never sampled from \(\mu\), and we may choose the conditional distribution
arbitrarily.

Given \(t\in[n]^r\) and \(a=(a_1,\ldots,a_r)\in\{0,1\}^r\), define
\[
        g_t(a)
\]
to be the conditional expectation, over these posterior labels, of the labelled
reconstruction described above, with the observed bits \(a_j\) substituted for
\(x_{t_j}\).  More explicitly, for each label tuple \(i\), if every label
\(1,\ldots,k\) occurs among \(i_1,\ldots,i_r\), let \(\tau_\ell(i)\) be the
first index \(j\) with \(i_j=\ell\), and set
\[
        Y_i(t,a)
        =
        f_{(t_{\tau_1(i)},\ldots,t_{\tau_k(i)})}
        (a_{\tau_1(i)},\ldots,a_{\tau_k(i)}).
\]
If not every label occurs, set \(Y_i(t,a)=0\).  Then define
\[
        g_t(a)
        =
        \mathbb E\bigl[ Y_I(t,a)\mid T=t\bigr].
\]
This is a deterministic function
\[
        g_t:\{0,1\}^r\to[0,1].
\]

For the actual input \(x\), the random variable
\[
        g_T(x_{T_1},\ldots,x_{T_r})
\]
is exactly
\[
        \mathbb E(Y\mid T_1,\ldots,T_r,x_{T_1},\ldots,x_{T_r}).
\]
Therefore, by Jensen's inequality,
\[
\begin{aligned}
        \left|
        g_T(x_{T_1},\ldots,x_{T_r})-f(x)
        \right|
        &=
        \left|
        \mathbb E(Y-f(x)\mid T_1,\ldots,T_r,x_{T_1},\ldots,x_{T_r})
        \right|                                                   \\
        &\leq
        \mathbb E\left(
        |Y-f(x)|
        \mid T_1,\ldots,T_r,x_{T_1},\ldots,x_{T_r}
        \right).
\end{aligned}
\]
Taking expectations gives
\[
        \mathbb E
        \left|
        g_T(x_{T_1},\ldots,x_{T_r})-f(x)
        \right|
        \leq
        \mathbb E |Y-f(x)|
        \leq
        2\alpha+\eta .
\]
By Markov's inequality,
\[
        \mathbb P\left(
        \left|
        g_T(x_{T_1},\ldots,x_{T_r})-f(x)
        \right|
        >\varepsilon
        \right)
        \leq
        \frac{2\alpha+\eta}{\varepsilon}.
\]
With our choices
\[
        \alpha=\frac{\varepsilon^2}{4},
        \qquad
        \eta=\frac{\varepsilon^2}{2},
\]
we have
\[
        2\alpha+\eta=\varepsilon^2.
\]
Hence
\[
        \mathbb P\left(
        \left|
        g_T(x_{T_1},\ldots,x_{T_r})-f(x)
        \right|
        >\varepsilon
        \right)
        \leq \varepsilon .
\]
This holds for every \(x\in\{0,1\}^n\).  Therefore \(f\) is
identically sampled \((r,\varepsilon)\)-holographic with identical sampling
from the single measure \(\mu\).
\end{proof}

\section*{Acknowledgements}
The author acknowledges the use of ChatGPT as a writing aid during the preparation of this manuscript, specifically for improving grammar, style, and exposition. All mathematical ideas, statements, and proofs were developed by the author. 
This research was supported by the NKFIH excellence 154126 grant.

\end{document}